\documentclass[10pt]{article}
\usepackage{amssymb,amsmath}

\def\R{\mathbb{R}}

\def\stac{\stackrel}
\def\~{\stac{\sim}}
\def\8{\infty}

\def\({\left(}
\def\){\right)}

\def\be{\begin{equation}}
\def\ee{\end{equation}}

\def\ba{\begin{array}}
\def\ea{\end{array}}

\def\({\left(}\def\){\right)}

\def\to{\rightarrow}

\begin{document}
\title{Upper and Lower Bounds for Blow-up Times in Coupled\\
Reaction--Diffusion Systems with Nonlocal Sources}
\author{Xianfa Song$^{a,b}$\thanks{email:  \texttt{songxianfa@tju.edu.cn} (or \texttt{songxianfa2004@163.com)}} Zeen Song$^{c}$\\
\small $^{a}$ Department of Mathematics, School of Mathematics, Tianjin University,\\[-0.2em]
\small Tianjin 300072, P. R. China 
\\
\small $^{b}$ Xinjiang Production and Construction Corps Key Laboratory of Green and Intelligent Development\\
\small and Efficient Utilization of Strategic Mineral Resources, Xinjiang University of Technology, Hotan
84800, China\\
\small $^{c}$ School of Mathematics, Statistics and Mechnics, Beijing University of Technology,\\
\small No. 100 Pingleyuan, Chaoyang District, Beijing, 100124, P. R. China}
\date{}

\maketitle

\begin{abstract}
In this paper, we consider the following reaction-diffusion system
\begin{equation*}
\begin{cases}
 u_t=\Delta u+F\bigl(x,t,u,v,V(t)\bigr), & (x,t)\in\Omega\times(0,T),\\
 v_t=\Delta v+G\bigl(x,t,u,v,U(t)\bigr), & (x,t)\in\Omega\times(0,T),\\
 \mathcal B u=\mathcal B v=0, & (x,t)\in\partial\Omega\times(0,T),\\
 u(x,0)=u_0(x),\quad v(x,0)=v_0(x), & x\in\Omega,
\end{cases}
\end{equation*}
where
\begin{equation*}
 U(t)=\int_{\Omega}u(x,t)dx,\qquad
 V(t)=\int_{\Omega}v(x,t)dx,
\end{equation*}
and the boundary operator $\mathcal B$ is either
$$
 \mathcal B_D w=w
 \qquad\text{or}\qquad
 \mathcal B_N w=\partial_\nu w,
$$
$\Omega\subset\mathbb{R}^N$ be a bounded domain with sufficiently smooth boundary, $N\ge 1$, and $\nu$ is the outward unit normal on $\partial\Omega$.

Under certain conditions, we obtain the upper and lower bounds for the blowup time of the solution to the system.
\end{abstract}



\section{Introduction}

\qquad In this paper, we consider the following reaction-diffusion system:
\begin{equation}\label{eq:system}
\begin{cases}
 u_t=\Delta u+F\bigl(x,t,u,v,V(t)\bigr), & (x,t)\in\Omega\times(0,T),\\
 v_t=\Delta v+G\bigl(x,t,u,v,U(t)\bigr), & (x,t)\in\Omega\times(0,T),\\
 \mathcal B u=\mathcal B v=0, & (x,t)\in\partial\Omega\times(0,T),\\
 u(x,0)=u_0(x),\quad v(x,0)=v_0(x), & x\in\Omega,
\end{cases}
\end{equation}
where
\begin{equation}\label{eq:masses}
 U(t)=\int_{\Omega}u(x,t)dx,\qquad
 V(t)=\int_{\Omega}v(x,t)dx,
\end{equation}
and the boundary operator $\mathcal{B}$ is either
\[
 \mathcal B_D w=w
 \qquad\text{or}\qquad
 \mathcal B_N w=\partial_\nu w,
\]
where $\Omega\subset\R^N$ is a bounded domain with sufficiently smooth boundary, $N\ge 1$, and $\nu$ is the outward unit normal on $\partial\Omega$. Under certain assumptions,  we can place the nonlocal system in the standard semilinear local well-posedness and continuation framework \cite{Henry1981,Amann1984} and establish the well-posedness result in Section 2.

 However, we are mainly concerned with the lower and upper bounds for the blowup time of the solution to the system. Finite time blowup is an interesting topics on nonlinear parabolic equations. We can refer to \cite{Kaplan1963, Fujita1966, FriedmanMcLeod1985,BandleBrunner1998, QuittnerSouplet2019} which studied
the finite time blowup phenomenon of the solution to the scalar  parabolic equation. For the systems of parabolic equations, we can refer to \cite{EscobedoHerrero1991, EscobedoLevine1995} and the references therein to see the results on blowup phenomenon.

A related quantitative problem is to estimate the bounds for blowup time. In \cite{BebernesKassoy1981}, Bebernes and Kassoy had given the explicit estimates for the classical thermal-runaway models. About the results on bound for the solution to scalar equation, we can refer to \cite{Payne2006,Payne2007, Payne2008,Payne20091, Payne20092, Song2011} and the references therein.  For the system of parabolic equations with nonlocal factors in reaction terms, lower bounds for blowup time were derived in \cite{TangLiYang2014},  upper and lower bounds for blowup time were obtained by Bao and Song in \cite{BaoSong2016} as well as the  related estimates for quasilinear systems in \cite{LuBaoSong2017,BaoSong2021}. Further results can be found in \cite{ MarrasVernierPiro2017, Xue2021, KwonFang2022,LiuZhangLiu2025,Liu2026}. Recently, the bounds for blowup time of stochastic system were discussed in \cite{SankarMohanKarthikeyan2024}.

We will estimate the bounds for blowup time under the following assumptions respectively.

{\bf Assumption 1.} There exist $p_1,q_1,p_2,q_2>1$, functions $a_1,a_2\in C(\overline\Omega\times[0,\infty))$, and constants
\[
 a_1^*,a_2^*,b_1^*,b_2^*>0
\]
such that $a_i(x,t)\ge a_i^*>0$ and
\begin{align}
F\bigl(x,t,u,v,V(t)\bigr)
&\ge a_1(x,t)v^{p_1}+b_1^*V(t)^{p_2},\label{eq:A1F}\\
G\bigl(x,t,u,v,U(t)\bigr)
&\ge a_2(x,t)u^{q_1}+b_2^*U(t)^{q_2}.\label{eq:A1G}
\end{align}

{\bf Assumption 2.} For some $\alpha,\beta>1$, $0<\rho,\sigma<1$, positive constants $c_1,c_2,C_1,C_2$, nonnegative constants $b_1,b_2$, and bounded nonnegative functions $A_1,A_2$ satisfying
\[
 0\le A_1(x,t)\le A_1^*,\qquad 0\le A_2(x,t)\le A_2^*,
\]
assume $0<c_i\le C_i$ and
\begin{align}
 c_1V(t)^\alpha
 &\le F\bigl(x,t,u,v,V(t)\bigr)
 \le C_1V(t)^\alpha+A_1(x,t)v^\rho+b_1,\label{eq:A2F}\\
 c_2U(t)^\beta
 &\le G\bigl(x,t,u,v,U(t)\bigr)
 \le C_2U(t)^\beta+A_2(x,t)u^\sigma+b_2.\label{eq:A2G}
\end{align}

The paper is organized as follows. In Section~\ref{sec:setting}, we first establish the local well-posedness result and give the definition of blowup in finite time. Then we develop the moment minimization and obtain the rigorous Dirichlet weighted-diffusion identity. In Sections~\ref{sec:a1}, we establish upper bounds for the blowup time under Assumptions 1 and Assumption 2 respectively. In Section~\ref{sec:lower}, we first establish lower bounds for the blowup time under Assumption 2. Then we can obtain two-side estimates and explain why Assumption 1 alone admits no uniform positive lower bound.

\section{Preliminary}\label{sec:setting}
\qquad First, we give the definition of classical solution to (\ref{eq:system}) below.

{\bf Definition 2.1.} {\it  A classical solution to (\ref{eq:system}) on $[0,T)$ means a pair $(u,v)$ with
\[
 u,v\in C(\overline\Omega\times[0,T))\cap C^{2,1}(\Omega\times(0,T))
\]
which satisfies the governing equations in $\Omega\times(0,T)$, as well as the compatible boundary conditions for $0<t<T$ and the initial data continuously at $t=0$.}

To establish the local well-posedness result, we will impose the following standing hypotheses.

[{\bf H1.}]
$F$ and $G$ are locally H\"older continuous in  $(x,t)\in\overline\Omega\times[0,T_0]$ for every $T_0>0$ and uniformly Lipschitz continuous in $(u,v,m)$ satisfying  $0\le u,v,m\le R$ for every $R>0$.

[{\bf H2.}]
For every admissible $x,t,u,v$, $ F(x,t,0,v,V)\ge 0$, $G(x,t,u,0,U)\ge 0$.

[{\bf H3.}]
$F$ is nondecreasing with respect to the $v$ component and its mass variable, and $G$ is nondecreasing with respect to the $u$ component and its mass variable when the remaining variables keep fixed admissible values.

[{\bf H4.}]
The functions $u_0,v_0\ge 0$ are not both identically zero, $u_0,v_0\in C^{2+\vartheta}(\overline\Omega)$ for some $0<\vartheta<1$, and satisfy the compatible boundary conditions.

Under these assumptions, we can establish the following local well-posedness results.

{\bf Proposition 2.1.} {\it 
Assume H1--H4. Then \eqref{eq:system} has a unique nonnegative classical solution $(u,v)$ on a maximal interval
$$
 [0,T_{max}),\qquad 0<T_{max}\le\infty.
$$
If $T_{max}<\infty$, then
\begin{equation}\label{eq:cont}
 \limsup_{t\uparrow T_{max}}
 \left(\|{u(\cdot,t)}\|_{L^\infty(\Omega)}+
 \|{v(\cdot,t)}\|_{L^\infty(\Omega)}\right)=\infty.
\end{equation}
}

{\bf Proof:}
The Dirichlet and Neumann Laplacians generate analytic positive heat semigroups on $C(\overline\Omega)$.
 Since $w\mapsto\int_\Omega w$ is a bounded linear functional on $C(\overline\Omega)$, the induced nonlocal reaction map is locally Lipschitz on bounded subsets of $C(\overline\Omega)^2$.
Combining it with H1-H4,  we can use standard semilinear parabolic theory to obtain local existence, uniqueness, and the blowup alternative, see \cite{Henry1981,Amann1984}.  As a result, the only obstruction to continuation at finite time is unboundedness of the $C(\overline\Omega)^2$ norm, which is exactly \eqref{eq:cont}.

{\bf Definition 2.2.}  When $T_{max}<\infty$, we call $T_{max}$ the blow-up time.  Especially, if there exists a finite time $T$ such that
$$
 \lim_{t\uparrow T}\bigl(U(t)+V(t)\bigr)=\infty,
$$
we say that the solution undergoes $L^1$ blow-up at time $T$.

We will establish the upper bounds for the blowup time of the solution to ODE's inequality below.

First, we give a moment minimization lemma as follows.

{\bf Lemma 2.3.} {\it
Let
$$
 P(y)=\sum_{i=1}^m a_i y^{\tilde{p}_i},\qquad
 Q(x)=\sum_{j=1}^n b_j z^{\tilde{q}_j},
$$
where $a_i,b_j>0$ and $\tilde{p}_i,\tilde{q}_j>1$. Define
\begin{equation}\label{eq:hdef}
 h(s)=\min_{0\le z\le s}\{P(s-z)+Q(z)\},\qquad s\ge0.
\end{equation}
Then:

(i) the minimum is attained for every $s\ge0$;

(ii) $h(0)=0$, $h(s)>0$ for $s>0$, and $h$ is continuous and locally Lipschitz on $[0,\infty)$;

(iii) $h(s)/s$ is nondecreasing on $(0,\infty)$;

(iv) there exist $c>0$, $r>1$, and $s_*>0$ such that $h(s)\ge cs^r$ for $s\ge s_*$.

Consequently,
$$
 \int_R^\infty \frac{d s}{h(s)}<\infty
$$
for every $R>0$.}

{\bf Proof:} Note that for fixed $s$, the objective in \eqref{eq:hdef} is continuous on the compact interval $[0,s]$, so the minimum value can be attained and it can vanish only at $s=0$.

Let $z=\theta s$ and
\[
 \mathcal F(s,\theta)=P((1-\theta)s)+Q(\theta s),
 \qquad (s,\theta)\in[0,\infty)\times[0,1].
\]
Then $h(s)=\min_{0\le\theta\le1}\mathcal F(s,\theta)$ and $h$ is continuous because uniform continuity of $\mathcal F$ on compact rectangles. Moreover, the family $\{\mathcal F(\cdot,\theta)\}_{\theta\in[0,1]}$ is equi-Lipschitz on each $[0,R]\times[0,1]$ because the $s$-derivatives of all monomial terms are uniformly bounded. As a result, $h(s)$ preserves the same Lipschitz bound.

For $s>0$, let
\[
 \frac{h(s)}s
 =\min_{0\le\theta\le1}
 \left[
 \sum_i a_i(1-\theta)^{\tilde{p}_i}s^{\tilde{p}_i-1}
 +\sum_j b_j\theta^{\tilde{q}_j}s^{\tilde{q}_j-1}
 \right].
\]
Since all exponents $\tilde{p}_i-1$ and $\tilde{q}_j-1$ are positive, the bracket is nondecreasing in $s$ for each fixed $\theta$. Consequently, its minimum is nondecreasing.

Denote $p_0=p_{i_0}$, $a_0=a_{i_0}$, $q=q_{j_0}$ and $b=b_{j_0}$ for arbitrarily chosen $i_0$ and $j_0$. Since for every $z\in[0,s]$, either $s-z\ge s/2$ or $z\ge s/2$, we have
\[
 P(s-z)+Q(z)\ge \min\{a2^{-p}s^p,b2^{-q}s^q\}.
\]
Especially, this is bounded below by $cs^r$ for $s\ge1$ with
\[
 r=\min\{p,q\}>1,
 \qquad c=\min\{a2^{-p},b2^{-q}\}>0.
\]
(iv) is proved and the Osgood integrability follows.\hfill $\Box$

{\bf Lemma 2.4.} {\it 
 Assume that $S:[0,T)\to(0,\infty)$ is a locally absolutely continuous satisfies
\begin{equation}\label{eq:osgdiff}
 S'(t)\ge-\lambda S(t)+h(S(t))\quad\text{for a.e. }t\in(0,T),
 \qquad S(0)=S_0>0.
\end{equation}
Here $h$ is defined in Lemma 2.3 and $\lambda\ge 0$ is a constant. If
\begin{equation}\label{eq:thresholdabstract}
 h(S_0)>\lambda S_0,
\end{equation}
then $S(t)\ge S_0$ and
\begin{equation}\label{eq:osgbound}
 T\le \int_{S_0}^{\infty}\frac{ds}{h(s)-\lambda s}.
\end{equation}}

{\bf Proof:} Since $h(s)/s$ is nondecreasing,
\[
 \frac{h(s)-\lambda s}{s}
 \ge \frac{h(S_0)}{S_0}-\lambda=:\delta>0,
 \qquad s\ge S_0.
\]
Substituting it into (\ref{eq:osgdiff}),  we can  $S'\ge\delta S_0>0$ almost everywhere.

Let $g(s)=h(s)-\lambda s>0$ for $s\ge S_0$ and
\[
 \Psi(r)=\int_{S_0}^{r}\frac{ds}{g(s)}.
\]
 By Lemma 2.3, $h(s)\ge cs^r$ with $r>1$ for large $s$, hence $g(s)\ge(c/2)s^r$ eventually, and $\Psi(\infty)<\infty$.

Using the chain rule for absolutely continuous functions, we obtain
\[
 \frac{d}{d t}\Psi(S(t))=\frac{S'(t)}{g(S(t))}\ge 1
\]
for almost every $t$. Integrating it, we obtain  $t\le\Psi(S(t))\le\Psi(\infty)<\infty$ for all $t<T$, and \eqref{eq:osgbound} as $t\uparrow T$.\hfill $\Box$

Let $\varphi_1$ and $\lambda_1$ is the first engipair
\begin{equation}\label{eq:eigen}
 -\Delta\varphi_1=\lambda_1\varphi_1\quad\text{in }\Omega,
 \qquad \varphi_1=0\quad\text{on }\partial\Omega,
\end{equation}
normalized by
\[
 \varphi_1>0,\qquad \int_\Omega\varphi_1dx=1.
\]
Set
\begin{equation}\label{eq:weightedmoments}
 X(t)=\int_\Omega u(x,t)\varphi_1(x)dx,
 \qquad
 Y(t)=\int_\Omega v(x,t)\varphi_1(x)dx,
 \qquad S(t)=X(t)+Y(t),
\end{equation}
and
\begin{align}
\|\varphi\|=\|{\varphi_1}\|_{L^\infty(\Omega)}.\label{1041}
\end{align}

We will give another lemma below.

{\bf Lemma 2.5.} {\it 
Under the Dirichlet boundary condition, we have
\begin{align}
 X'(t)&=-\lambda_1X(t)+\int_\Omega F(t)\varphi_1dx,\label{eq:Xprime}\\
 Y'(t)&=-\lambda_1Y(t)+\int_\Omega G(t)\varphi_1dx.\label{eq:Yprime}
\end{align}
}

{\bf Proof:}
Let $S_D(t)$ be the Dirichlet heat semigroup. By Duhamel's formula, we get
\[
 u(t)=S_D(t-s)u(s)+\int_s^t S_D(t-\tau)F(\tau)d\tau
\]
for $0\le s<t<T_{max}$.
The semigroup is self-adjoint on $L^2(\Omega)$ and $S_D(t)\varphi_1=e^{-\lambda_1 t}\varphi_1$. Therefore,
\[
 X(t)=e^{-\lambda_1(t-s)}X(s)
 +\int_s^t e^{-\lambda_1(t-\tau)}
 \left(\int_\Omega F(\tau)\varphi_1dx\right)d\tau,
\]
which implies \eqref{eq:Xprime}. \eqref{eq:Yprime} can  be obtained similarly.  \hfill $\Box$

Using the results above, we can give the estimate for $T_{max}$ under certain conditions

{\bf Theorem 2.6.}{\it 
Suppose that nonnegative moments $X,Y$ satisfy
\[
 X'\ge-\lambda X+P(Y),\qquad
 Y'\ge-\lambda Y+Q(X),
\]
where $\lambda>0$ and $P,Q$ are defined in Lemma 2.3.
If $h(S_0)>\lambda S_0$, then
\begin{equation}\label{eq:abstractDbound}
 T_{max}\le \int_{S_0}^\infty\frac{ds}{h(s)-\lambda s},
\end{equation}
where $h(s)$ is defined by (\ref{eq:hdef}).
}

{\bf Proof:} Adding the two inequalities, we get
\[
 S'\ge-\lambda S+P(Y)+Q(X)\ge-\lambda S+h(S).
\]
Using Lemma 2.4 and Lemma 2.5, we get the conclusion.\hfill $\Box$

\section{Upper bounds for the blowup time}\label{sec:a1}

\qquad In this section, we will establish the upper bounds for the blowup time under Assumption 1 and Assumption 2 respectively. Denote $D=|\Omega|$ and $W(t)=U(t)+V(t)$.
We will discuss it in four cases.

{\bf Case 1.} Under Assumption 1 and homogeneous Neumann boundary conditions. Note that
\[
 U'(t)=\int_\Omega Fdx,\qquad V'(t)=\int_\Omega Gdx.
\]

Using Jensen's inequality, we can get
\begin{align}
 U'&\ge a_1^*D^{1-p_1}V^{p_1}+b_1^*D V^{p_2},\label{eq:A1NU}\\
 V'&\ge a_2^*D^{1-q_1}U^{q_1}+b_2^*D U^{q_2}.\label{eq:A1NV}
\end{align}
Define
\begin{align*}
 P_{1,N}(y)&=a_1^*D^{1-p_1}y^{p_1}+b_1^*Dy^{p_2},\\
 Q_{1,N}(z)&=a_2^*D^{1-q_1}z^{q_1}+b_2^*Dz^{q_2},
\end{align*}
and
\begin{equation}\label{eq:h1N}
 h_{1,N}(s)=\min_{0\le z\le s}\{P_{1,N}(s-z)+Q_{1,N}(z)\}.
\end{equation}
Then we have

{\bf Theorem 3.1.} {\it
Assume that H1--H4 hold. Then under Assumption 1 and homogeneous Neumann boundary conditions,  for every nonzero nonnegative initial datum, the solution has finite maximal existence time and
\begin{equation}\label{eq:A1Nbound}
 T_{max}\le\int_{W_0}^{\infty}\frac{ds}{h_{1,N}(s)},
 \qquad W_0=U(0)+V(0)>0.
\end{equation}
}

{\bf Proof:} Using \eqref{eq:A1NU}--\eqref{eq:A1NV}, we obtain
\[
 W'\ge P_{1,N}(V)+Q_{1,N}(U)\ge h_{1,N}(W).
\]
Since $W_0>0$, by the results of  Lemmas 2.3 and Lemmas 2.4 with $\lambda=0$,  $W$ cannot remain finite beyond the right-hand side of \eqref{eq:A1Nbound}. That is, every classical solution has finite mass before its maximal time $T_{max}$ which is bounded by the finite integral in \eqref{eq:A1Nbound}.
\hfill $\Box$

{\bf Case 2.} Under Assumption 2 and homogeneous Dirichlet boundary conditions.

Since $\varphi_1dx$ is a probability measure, using weighted Jensen's inequality, we
have
\[
 \int_\Omega v^{p_1}\varphi_1dx\ge Y^{p_1},
 \qquad
 \int_\Omega u^{q_1}\varphi_1dx\ge X^{q_1}.
\]
Moreover,
\[
 Y\le\|V\|,\qquad X\le\|U\|,\qquad V\ge\|\varphi\|^{-1}Y, \qquad U\ge\|\varphi\|^{-1}X.
\]
Using Lemma 2.5, we get
\begin{align*}
 X'&\ge-\lambda_1X+a_1^*Y^{p_1}+b_1^*\|\varphi\|^{-p_2}Y^{p_2},\\
 Y'&\ge-\lambda_1Y+a_2^*X^{q_1}+b_2^*\|\varphi\|^{-q_2}X^{q_2}.
\end{align*}
Define
\begin{align}
 h_{1,D}(s)=\min_{0\le x\le s}\{&a_1^*(s-x)^{p_1}
 +b_1^*\|\varphi\|^{-p_2}(s-x)^{p_2}\notag\\
 &+a_2^*x^{q_1}+b_2^*\|\varphi\|^{-q_2}x^{q_2}\}.\label{eq:h1D}
\end{align}

{\bf Theorem 3.2.} {\it 
Assume thatH1--H4 hold. Then under Assumption 1 and homogeneous Dirichlet boundary conditions,  for every nonzero nonnegative initial datum, the solution has finite maximal existence time. If
\begin{equation}\label{eq:A1Dthreshold}
 h_{1,D}(S_0)>\lambda_1S_0,
\end{equation}
then $T_{max}<\infty$ and
\begin{equation}\label{eq:A1Dbound}
 T_{max}\le\int_{S_0}^{\infty}
 \frac{ds}{h_{1,D}(s)-\lambda_1s}.
\end{equation}
}

{\bf Proof:}
Applying Theorem 3.1 to the two weighted-moment inequalities above, we can obtian the conclusion.
\hfill $\Box$

{\bf Case 3. } Under Assumption 2 and homogeneous Neumann boundary conditions.

Under Assumption 2, we have
\[
 U'\ge Dc_1V^\alpha,\qquad V'\ge Dc_2U^\beta.
\]
Define
\begin{equation}\label{eq:h2N}
 h_{2,N}(s)=D\min_{0\le z\le s}
 \{c_1(s-z)^\alpha+c_2z^\beta\}.
\end{equation}

Then we have

{\bf Theorem 3.3. } {\it
Assume that H1--H4 hold. Then under Assumption 2 and homogeneous Neumann boundary conditions, for every nonzero nonnegative initial datum, the solution has finite maximal existence time and
\begin{equation}\label{eq:A2Nbound}
 T_{max}\le\int_{W_0}^{\infty}\frac{ds}{h_{2,N}(s)}.
\end{equation}
}

{\bf Proof:} Since $W'\ge h_{2,N}(W)$ and $W_0>0$, using Lemmas 2.3 and 2.4, we can get \eqref{eq:A2Nbound}.
\hfill $\Box$

{\bf Case 4. } Under Assumption 2 and homogeneous Dirichlet boundary conditions

Under Assumption 2, using Lemma 2.5, we get
\[
 X'\ge-\lambda_1X+c_1\|\varphi\|^{-\alpha}Y^\alpha,
 \qquad
 Y'\ge-\lambda_1Y+c_2\|\varphi\|^{-\beta}X^\beta.
\]
Let
\[
 A=c_1\|\varphi\|^{-\alpha},\qquad B=c_2\|\varphi\|^{-\beta},
\]
and
\begin{equation}\label{eq:h2D}
 h_{2,D}(s)=\min_{0\le z\le s}\{A(s-z)^\alpha+Bz^\beta\}.
\end{equation}

{\bf Theorem 3.4}
Assume that H1--H4 hold. Then under Assumption 2 and homogeneous Neumann boundary conditions, for every nonzero nonnegative initial datum, the solution has finite maximal existence time. If
\begin{equation}\label{eq:A2Dthreshold}
 h_{2,D}(S_0)>\lambda_1S_0,
\end{equation}
then $T_{max}<\infty$ and
\begin{equation}\label{eq:A2Dbound}
 T_{max}\le\int_{S_0}^{\infty}
 \frac{ds}{h_{2,D}(s)-\lambda_1s}.
\end{equation}

{\bf Proof:}
Applying Theorem 3.3, we can obtain the conclusion.
\hfill $\Box$

Especially, when the two mass exponents in Assumption 2 coincide, we can obtain some closed forms. To do this, we need the following lemma

{\bf Lemma 3.5} {\it
Let $A,B>0$ and $r>1$. Then
\begin{equation}\label{eq:eqmin}
 \min_{0\le z\le s}\{A(s-z)^r+Bz^r\}
 =\left(A^{-1/(r-1)}+B^{-1/(r-1)}\right)^{-(r-1)}s^r.
\end{equation}
For $s>0$ the minimizer is unique and satisfies
\[
 Bz^{r-1}=A(s-z)^{r-1}.
\]
}

{\bf Proof:}  By strict convexity, there exists a unique critical point. Let $a=A^{-1/(r-1)}$ and $b=B^{-1/(r-1)}$. The critical-point relation yields
\[
 s-z=\frac{a}{a+b}s,\qquad z=\frac{b}{a+b}s.
\]
Substituting them into (\ref{eq:eqmin}), we get the conclusion.\hfill $\Box$

Now we give a corollary as follows.

{\bf Corollary 3.6}
Assume that $\alpha=\beta=r>1$ and define
\begin{equation}\label{eq:CN}
 C_N=\left(c_1^{-1/(r-1)}+c_2^{-1/(r-1)}\right)^{-(r-1)}.
\end{equation}
Then under homogeneous Neumann boundary conditions,
\[
 h_{2,N}(s)=DC_Ns^r,
\]
and
\begin{equation}\label{eq:eqNbound}
 T_{max}\le\frac{W_0^{1-r}}{(r-1)DC_N}.
\end{equation}
Under homogeneous Dirichlet boundary conditions, let
\begin{equation}\label{eq:CD}
 C_D=\|\varphi\|^{-r}C_N.
\end{equation}
If
\begin{equation}\label{eq:eqDthreshold}
 C_DS_0^{r-1}>\lambda_1,
\end{equation}
then
\begin{equation}\label{eq:eqDbound}
 T_{max}\le
 \frac{1}{(r-1)\lambda_1}
 \log\!\left(
 \frac{C_DS_0^{r-1}}{C_DS_0^{r-1}-\lambda_1}
 \right).
\end{equation}

{\bf Proof:}
Since the definitions of $h_{2,N}$ and $h_{2,D}$ have been defined in Lemma 3.5. The Neumann integral is elementary. In the Dirichlet case,
\[
 \int_{S_0}^{\infty}\frac{ds}{C_Ds^r-\lambda_1s}
 =\frac{1}{(r-1)\lambda_1}
 \log\!\left(
 \frac{C_DS_0^{r-1}}{C_DS_0^{r-1}-\lambda_1}
 \right),
\]
provided \eqref{eq:eqDthreshold} holds.
\hfill $\Box$

\section{Lower bounds for blowup time under Assumption 2}\label{sec:lower}

\qquad In this section, we first establish lower bounds for blowup time by using the right parts of inequalities in Assumption 2, then we can obtain
two-side bounds for the blowup time by combining the left parts of inequalities in Assumption 2  as well as the results in Theorem 3.3 and Theorem 3.4.

Let
\[
 \mathcal U(t)=\|u(\cdot,t)\|_{L^\infty(\Omega)},\qquad
 \mathcal V(t)=\|v(\cdot,t)\|_{L^\infty(\Omega)},\qquad
 Z(t)=\mathcal U(t)+\mathcal V(t).
\]
Denote $S_B(t)$ the heat semigroup for either homogeneous Neumann or homogeneous Dirichlet boundary conditions. It is positivity preserving and an $L^\infty$ contraction.  Using Duhamel's formula and Assumption 2, we get
\[
 D^+\mathcal U(t)\le\|F(t)\|_\infty,
 \qquad
 D^+\mathcal V(t)\le\|G(t)\|_\infty
\]
for $0<h<T_{max}-t$, where $D^+$ denotes the upper right Dini derivative. Since $V\le D\mathcal V\le DZ$ and $U\le D\mathcal U\le DZ$, we have
\begin{align*}
 \|F(t)\|_\infty&\le C_1D^\alpha Z^\alpha+A_1^*Z^\rho+b_1,\\
 \|G(t)\|_\infty&\le C_2D^\beta Z^\beta+A_2^*Z^\sigma+b_2.
\end{align*}
Define
\begin{equation}\label{eq:Hinf}
 H_\infty(s)=C_1D^\alpha s^\alpha+C_2D^\beta s^\beta
 +A_1^*s^\rho+A_2^*s^\sigma+b_1+b_2.
\end{equation}
Then
\begin{equation}\label{eq:ZDini}
 D^+Z(t)\le H_\infty(Z(t)).
\end{equation}

And we have

{\bf Theorem 4.1} {\it 
Assume that H1--H4 hold. Then under Assumption 2, with either homogeneous Dirichlet or homogeneous Neumann boundary conditions,
\begin{equation}\label{eq:lowerLinf}
 T_{max}\ge\int_{Z_0}^{\infty}\frac{ds}{H_\infty(s)},
 \qquad
 Z_0=\|u_0\|_\infty+\|v_0\|_\infty>0.
\end{equation}
}

{\bf Proof:} Let $z$ solve
\[
 z'=H_\infty(z),\qquad z(0)=Z_0.
\]
Note that the function $H_\infty$ is continuous, locally Lipschitz, positive, and nondecreasing on $[0,\infty)$. Applying the standard Dini-derivative comparison to \eqref{eq:ZDini}, we get that $Z(t)\le z(t)$ while both are finite. The scalar blowup time of $z$ is
\[
 \tau_\infty=\int_{Z_0}^{\infty}\frac{ds}{H_\infty(s)},
\]
which is finite because $\alpha,\beta>1$. If $T_{max}<\tau_\infty$, then $Z(t)$ would remain bounded as $t\uparrow T_{max}$, contradicting the continuation criterion \eqref{eq:cont}. Hence \eqref{eq:lowerLinf} holds.
\hfill $\Box$

We will estimate for the lower bound for the blowup time of the solution when its $L^1$ norm become unbounded in finite time. Note that  diffusion conserves total mass under Neumann boundary conditions, while  the positive Dirichlet heat semigroup is sub-Markovian and decreases total mass of nonnegative data under Dirichlet boundary conditions. Since Assumption 2 implies that $F,G\ge0$, using Duhamel's formula we get
\[
 D^+W(t)\le H_1(W(t))
\]
for either boundary condition, where
\begin{equation}\label{eq:H1}
 H_1(s)=C_1Ds^\alpha+C_2Ds^\beta
 +A_1^*D^{1-\rho}s^\rho+A_2^*D^{1-\sigma}s^\sigma
 +(b_1+b_2)D.
\end{equation}

Then we can get

{\bf Proposition 4.2} {\it 
Assume that H1--H4 hold.  Under Assumption 2, with either homogeneous Dirichlet or homogeneous Neumann boundary conditions, if the maximal solution undergoes $L^1$ blow-up at a finite time $T$, then
\begin{equation}\label{eq:l1lower}
 T\ge\int_{W_0}^{\infty}\frac{ds}{H_1(s)}.
\end{equation}
}

{\bf Proof:} Let $w$ solve $w'=H_1(w)$, $w(0)=W_0$. Using Dini's comparison principle, we get $W(t)\le w(t)$ when $w$ is finite. The scalar blowup time is bounded by the right-hand side of \eqref{eq:l1lower}.
\hfill $\Box$

Combining the results in Section 3 with Theorem 4.1, we can immediately obtain two-sided estimates below.

{\bf Corollary 4.3} {\it
Assume H1--H4 and Assumption 2 hold. Then

(i) Under homogeneous Neumann boundary conditions,
\begin{equation}\label{eq:twosidedN}
 \int_{Z_0}^{\infty}\frac{ds}{H_\infty(s)}
 \le T_{max}
 \le \int_{W_0}^{\infty}\frac{ds}{h_{2,N}(s)}.
\end{equation}
(ii) Under homogeneous Dirichlet boundary conditions, if $h_{2,D}(S_0)>\lambda_1S_0$, then
\begin{equation}\label{eq:twosidedD}
 \int_{Z_0}^{\infty}\frac{ds}{H_\infty(s)}
 \le T_{max}
 \le \int_{S_0}^{\infty}\frac{ds}{h_{2,D}(s)-\lambda_1s}.
\end{equation}
If, in addition, $L^1$ blow-up occurs at the maximal time, Proposition 4.2 supplies the further lower bound \eqref{eq:l1lower}.
}

Last, for companion, we will show that  Assumption 1 cannot by itself provide any uniform positive lower bound for $T_{max}$ in terms of the displayed lower-growth data.

{\bf Proposition 4.4} {\it
Fixing a domain, exponents, the coefficients in Assumption 1 and a positive constant initial datum for the Neumann problem, then there is no positive lower bound for $T_{max}$ that depends only on these data and is valid for every reaction pair satisfying Assumption 1.
}

{\bf Proof:} If we take $u_0=v_0=m>0$, fix $r>1$ and set
\begin{align*}
 F_K&=a_1^*v^{p_1}+b_1^*V^{p_2}+Ku^r,\\
 G_K&=a_2^*u^{q_1}+b_2^*U^{q_2}+Kv^r
\end{align*}
for every $K>0$, then these reactions satisfy Assumption 1. And we have
\[
 u'(t)\ge Ku(t)^r,
\]
consequently,
\[
 T_{\max}^{(K)}\le\frac{m^{1-r}}{(r-1)K}\longrightarrow0
 \qquad(K\to\infty).
\]
Thus no positive lower bound can depend only on the Assumption  1 with lower-growth data and this fixed initial datum.

Similarly, in the case of Dirichlet boundary condition, adding $Ku^r$ and testing against $\varphi_1$, one can obtain $X'\ge-\lambda_1X+KX^r$, whose Osgood upper time tends to zero as $K\to\infty$ once $KX_0^{r-1}>\lambda_1$. \hfill $\Box$

Last, we summarize the main results according to\\

 ``Regime,\quad   Finite-time conclusion,\quad  Upper bound,\quad   Lower bound ".

1. Assumption 1 + Neumann,  Every nonzero nonnegative datum, $\displaystyle\int_{W_0}^{\infty}h_{1,N}(s)^{-1}ds$, No uniform positive lower bound from Assumption 1 alone; \\

2. Assumption 1 + Dirichlet, If $h_{1,D}(S_0)>\lambda_1S_0$,  $\displaystyle\int_{S_0}^{\infty}\frac{ds}{h_{1,D}(s)-\lambda_1s}$, No uniform positive lower bound from Assumption 1 alone; \\

3. Assumption 2 + Neumann, Every nonzero nonnegative datum, $\displaystyle\int_{W_0}^{\infty}h_{2,N}(s)^{-1}ds$, $\displaystyle\int_{Z_0}^{\infty}H_\infty(s)^{-1}ds$; \\

4. Assumption 2 + Dirichlet, If $h_{2,D}(S_0)>\lambda_1S_0$, $\displaystyle\int_{S_0}^{\infty}\frac{ds}{h_{2,D}(s)-\lambda_1s}$, $\displaystyle\int_{Z_0}^{\infty}H_\infty(s)^{-1}ds$.

\end{document}